\documentclass[12pt]{amsart}
\usepackage{amsfonts,amssymb,amscd,amsmath,enumerate,color, xypic}
\def\NZQ{\mathbb}               

\def\ZZ{{\NZQ Z}}
\def\RR{{\NZQ R}}

\def\frk{\mathfrak}               

\def\Phi{{\frk N}}
\def\ab{{\bold a}}

\def\eb{{\bold e}}

\def\xb{{\bold x}}
\def\yb{{\bold y}}
\def\zb{{\bold z}}

\def\opn#1#2{\def#1{\operatorname{#2}}} 
\opn\chara{char} 
\opn\length{\ell} 
\opn\pd{pd} 
\opn\rk{rk}
\opn\projdim{proj\,dim} 
\opn\injdim{inj\,dim} 
\opn\rank{rank}
\opn\depth{depth} 
\opn\grade{grade} 
\opn\height{height}
\opn\embdim{emb\,dim} 
\opn\codim{codim}

\opn\Tr{Tr} 
\opn\bigrank{big\,rank}
\opn\superheight{superheight}
\opn\lcm{lcm}
\opn\trdeg{tr\,deg}
\opn\reg{reg} 
\opn\lreg{lreg} 
\opn\ini{in} 
\opn\lpd{lpd}
\opn\size{size}
\opn\mult{mult}
\opn\dist{dist}
\opn\cone{cone}
\opn\lex{lex}
\opn\rev{rev}
\opn\div{div} \opn\Div{Div} \opn\cl{cl} \opn\Cl{Cl}
\opn\Spec{Spec} \opn\Supp{Supp} \opn\supp{supp} \opn\Sing{Sing}
\opn\Ass{Ass} \opn\Min{Min}
\opn\Ann{Ann} \opn\Rad{Rad} \opn\Soc{Soc}
\opn\Syz{Syz} \opn\Im{Im} \opn\Ker{Ker} \opn\Coker{Coker}
\opn\Am{Am} \opn\Hom{Hom} \opn\Tor{Tor} \opn\Ext{Ext}
\opn\End{End} \opn\Aut{Aut} \opn\id{id} \opn\ini{in}

\opn\nat{nat}
\opn\pff{pf}
\opn\Pf{Pf} \opn\GL{GL} \opn\SL{SL} \opn\mod{mod} \opn\ord{ord}
\opn\Gin{Gin}
\opn\Hilb{Hilb}\opn\adeg{adeg}\opn\std{std}\opn\ip{infpt}
\opn\Pol{Pol}
\opn\sat{sat}
\opn\Var{Var}
\opn\Gen{Gen}

\opn\aff{aff} \opn\con{conv} \opn\relint{relint} \opn\st{st}
\opn\lk{lk} \opn\cn{cn} \opn\core{core} \opn\vol{vol}
\opn\link{link} \opn\star{star}
\opn\gr{gr}

\def\Ac{{\mathcal A}}

\def\Hc{{\mathcal H}}

\def\Jc{{\mathcal J}}

\def\Fc{{\mathcal F}}
\def\Oc{{\mathcal O}}
\def\Pc{{\mathcal P}}
\def\Qc{{\mathcal Q}}

\def\Cc{{\mathcal C}}
\def\Mc{{\mathcal M}}

\def\vol{{\textnormal{vol}}}
\def\conv{{\textnormal{conv}}}
\def\pot#1#2{#1[\kern-0.28ex[#2]\kern-0.28ex]}

\opn\dirlim{\underrightarrow{\lim}}
\opn\inivlim{\underleftarrow{\lim}}
\def\Implies{\ifmmode\Longrightarrow \else
	\unskip${}\Longrightarrow{}$\ignorespaces\fi}
\def\implies{\ifmmode\Rightarrow \else
	\unskip${}\Rightarrow{}$\ignorespaces\fi}
\def\iff{\ifmmode\Longleftrightarrow \else
	\unskip${}\Longleftrightarrow{}$\ignorespaces\fi}

\let\:=\colon
\newtheorem{Theorem}{Theorem}[section]
\newtheorem{Lemma}[Theorem]{Lemma}
\newtheorem{Corollary}[Theorem]{Corollary}
\newtheorem{Proposition}[Theorem]{Proposition}
\newtheorem{Remark}[Theorem]{Remark}

\newtheorem{Example}[Theorem]{Example}

\newtheorem*{acknowledgement}{Acknowledgment}
\let\epsilon\varepsilon
\let\phi=\varphi
\let\kappa=\varkappa
\def\qed{\ifhmode\textqed\fi
	\ifmmode\ifinner\quad\qedsymbol\else\dispqed\fi\fi}
\def\textqed{\unskip\nobreak\penalty50
	\hskip2em\hbox{}\nobreak\hfil\qedsymbol
	\parfillskip=0pt \finalhyphendemerits=0}
\def\dispqed{\rlap{\qquad\qedsymbol}}

\opn\dis{dis}
\opn\height{height}
\opn\dist{dist}
\def\pnt{{\raise0.5mm\hbox{\large\bf.}}}

\opn\Lex{Lex}
\opn\conv{conv}

\begin{document}
\title{volume, facets and  dual polytopes of twinned chain polytopes}
\author[A. Tsuchiya]{Akiyoshi Tsuchiya}
\address[Akiyoshi Tsuchiya]{Department of Pure and Applied Mathematics,
	Graduate School of Information Science and Technology,
	Osaka University,
	Suita, Osaka 565-0871, Japan}
\email{a-tsuchiya@ist.osaka-u.ac.jp}

\subjclass[2010]{52B05, 52B20}
\date{}
\keywords{Gorenstein Fano polytope, reflexive polytope, order polytope, chain polytope, volume, facet, dual polytope}

\begin{abstract}
    Let $(P,\leq_P)$ and $(Q,\leq_Q)$ be finite partially ordered sets with $|P|=|Q|=d$,
    and $\mathcal{C}(P) \subset \mathbb{R}^d$ and $\mathcal{C}(Q) \subset \mathbb{R}^d$
    their chain polytopes.
    The twinned chain polytope of $P$ and $Q$ is the lattice polytope $\Gamma(\mathcal{C}(P),\mathcal{C}(Q)) \subset \mathbb{R}^d$ 
    which is the convex hull of $\mathcal{C}(P) \cup (-\mathcal{C}(Q))$. 
    It is known that twinned chain polytopes are  Gorenstein Fano polytopes with the integer decomposition property.
   	In the present paper, we study combinatorial properties of twinned chain polytopes.
	First, we will give the formula of the volume of twinned chain polytopes in terms of the underlying partially ordered sets.
	Second, we will identify the facet-supporting hyperplanes of twinned chain polytopes in terms of the underlying partially ordered sets.
	Finally, we will provide the vertex representations of the dual polytopes of twinned chain polytopes.
\end{abstract} 

\maketitle
\section*{introduction}
A convex polytope $\Pc \subset \RR^{d}$ is {\em integral} if all vertices belong to $\ZZ^{d}$. 
We say that an integral convex polytope $\Pc \subset \RR^{d}$ possesses the {\em integer decomposition property} if, for each integer $N > 0$ 
and for each ${\bf a} \in N \Pc \cap \ZZ^{d}$, there exist ${\bf a}_1, \ldots, {\bf a}_N \in \Pc \cap \ZZ^{d}$ 
such that ${\bf a} = {\bf a}_1 + \cdots +{\bf a}_N$, where $N \Pc = \{ N \alpha \mid \alpha \in \Pc \}$. 
Furthermore, an integral convex polytope $\Pc \subset \RR^{d}$ is {\em Fano} if the origin of $\RR^{d}$ 
is a unique integer point belonging to the interior of $\Pc$. 
A Fano polytope $\Pc \subset \RR^{d}$ is {\em Gorenstein} 
if its dual polytope 
\[
\Pc^{\vee} := \{ {\bf x} \in \RR^{d} \mid \langle {\bf x}, {\bf y} \rangle \le 1 \  \mathrm{for\ all}  \ {\bf y} \in \Pc \}
\]
is integral as well. 
A Gorenstein Fano polytope is also known as a reflexive polytope. 
In recent years, the study of Gorenstein Fano polytopes with the integer decomposition property  has become increasingly popular.
It is known that Gorenstein Fano polytopes correspond to Gorenstein toric Fano varieties,
and they are related to mirror symmetry (see, e.g., \cite{mirror,Cox}).
Moreover, the integer decomposition property is
particularly important in the theory and application of integer
programing \cite[\S 22.10]{integer}.

A finite \textit{partially ordered set} is a finite set $P$ with a reflexive, transitive, and anti-symmetric relation $\leq_{P}$.
Let $P = \{p_1, \ldots, p_d\}$.
A {\em linear extension} of $P$ is a permutation $\sigma = i_1 i_2 \cdots i_d$ of $[d] = \{1, 2, \ldots, d\}$ 
which satisfies $a < b$ if $p_{i_a} <_{P} p_{i_b}$. 
A subset $I$ of $P$ is called a {\em poset ideal} of $P$ if $p_{i} \in I$ and $p_{j} \in P$ together with $p_{j} \leq_{P} p_{i}$ guarantee $p_{j} \in I$.  
Note that the empty set $\emptyset$ and $P$ itself are poset ideals of $P$. 
Let $\Jc(P)$ denote the set of poset ideals of $P$.
A subset $A$ of $P$ is called an {\em antichain} of $P$ if all
$p_{i}$ and $p_{j}$ belonging to $A$ with $i \neq j$ are incomparable.  
In particular, the empty set $\emptyset$ and each 1-element subset $\{p_j\}$ are antichains of $P$.
Let $\Ac(P)$ denote the set of antichains of $P$.
For each subset $I \subset P$, 
we define the $(0, 1)$-vectors $\rho(I) = \sum_{p_{i}\in I} \eb_{i}$, 
where $\eb_{1}, \ldots, \eb_{d}$ are the canonical unit coordinate vectors of $\RR^{d}$.  
In particular $\rho(\emptyset)$ is the origin ${\bf 0}$ of $\RR^{d}$. 
In \cite{Stanley}, Stanley introduced the order polytope $\Oc(P)$ and the chain polytope $\Cc(P)$. 
It is known that $\Oc(P)$ and $\Cc(P)$ are the $d$-dimensional convex polytopes defined by
\[
\Oc(P) = \text{conv}(\{ \rho(I) \mid I \in \Jc(P) \}), 
\]
\[
\Cc(P) = \text{conv}(\{ \rho(A) \mid A \in \Ac(P) \}).
\]
Moreover, $\Oc(P)$ and $\Cc(P)$ have the same volume (\cite[Theorem 4.1]{Stanley}). 
In particular, the volume of $\Oc(P)$ and $\Cc(P)$ are equal to $e(P) / d!$, 
where $e(P)$ is the number of linear extensions of $P$ (\cite[Corollary 4.2]{Stanley}).

For two integral convex polytopes $\Pc,\Qc \subset \RR^d$, we set
$$\Gamma(\Pc,\Qc)=\text{conv}(\Pc \cup (-\Qc)) \subset \RR^d.$$
Let $(P,\leq_{P})$ and $(Q,\leq_{Q})$ be finite partially ordered sets with $|P|=|Q|=d$.
Then we can construct three integral convex polytopes $\Gamma(\Oc(P),  \Oc(Q)),\Gamma(\Oc(P),  \Cc(Q))$ and $\Gamma(\Cc(P),  \Cc(Q))$.
In particular, we call  $\Gamma(\Cc(P), \Cc(Q))$  the \textit{twinned chain polytope} of $P$ and $Q$.
 It is known $\Gamma(\Cc(P),  \Cc(Q))$ and $\Gamma(\Oc(P),\Cc(Q))$ are Gorenstein Fano polytopes with the integer decomposition property (\cite{orderchain,harmony}).
Furthermore if $P$ and $Q$ have a common linear extension, $\Gamma(\Oc(P), \Oc(Q))$ is a Gorenstein Fano polytope with the integer decomposition property (\cite{twin}).
In addition, if $P$ and $Q$ have a common linear extension, these three polytopes have the same volume (\cite{volOC}).

For two integral convex polytopes $\Pc,\Qc \subset \RR^d$, we set
$$\Omega(\Pc,\Qc)=\text{conv}(\Pc \times \{1\} \cup ((-\Qc) \times \{-1\})) \subset \RR^{d+1}.$$
In \cite{HT}, Hibi and Tsuchiya study the combinatorial properties of the three integral convex polytopes $\Omega(\Oc(P),  \Oc(Q))$, $\Omega(\Oc(P),  \Cc(Q))$ and $\Omega(\Cc(P),  \Cc(Q))$.
Moreover, in \cite{double}, Chappell, Friedl and
Sanyal study the combinatorial properties of the three integral convex polytopes $\Omega(2\Oc(P),  2\Oc(Q))$, $\Omega(2\Oc(P),  2\Cc(Q))$ and $\Omega(2\Cc(P),  2\Cc(Q))$.
These polytopes have strong connections with $\Gamma(\Oc(P),  \Oc(Q)),\Gamma(\Oc(P),  \Cc(Q))$ and $\Gamma(\Cc(P),  \Cc(Q))$.
For example, we can compute the volume of $\Omega(\Cc(P),  \Cc(Q))$ by computing that of $\Gamma(\Cc(P),  \Cc(Q))$.
In addition, we note that $\Gamma(\Cc(P),  \Cc(Q))$ arises as projection of $\Omega(2\Cc(P),  2\Cc(Q))$.
Hence, it is important to study properties of $\Gamma(\Oc(P),  \Oc(Q)),\Gamma(\Oc(P),  \Cc(Q))$ and $\Gamma(\Cc(P),  \Cc(Q))$.

In the present paper, we study combinatorial properties of twinned chain polytopes.
One of the reasons why we study twinned chain polytopes is that twinned chain polytopes have a nice structure.
In fact, we will show each twinned chain polytope of dimension $d$ is the union of $d$ chain polytopes (Lemma \ref{volumelemma} and Proposition \ref{prop}). 
In section 1, we will use this proposition to give the formula of the volume of a twinned chain polytope $\Gamma(\Cc(P), \Cc(Q))$ in terms of underlying partially ordered sets (Theorem \ref{volume}).
Similarly, we can compute the volume of $\Gamma(\Oc(P), \Oc(Q))$ and $\Gamma(\Oc(P),  \Cc(Q))$ (Corollary \ref{oo}).
By using Proposition \ref{prop} again, in Section 2, we will identify the facet-supporting hyperplanes of twinned chain polytopes in terms of the underlying partially ordered sets
(Theorem \ref{facetth}).
Finally, we will provide the vertex representations of the dual polytopes of twinned chain polytopes 
(Corollary \ref{cor}).

 \section{the formula of the volume of twinned chain polytopes}
 For partially ordered sets $(P,\leq_{P})$ and $(Q,\leq_{Q})$ with $P \cap Q =\emptyset$,
 the \textit{ordinal sum} of $P$ and $Q$ is the partially ordered set $(P \oplus Q,\leq_{P \oplus Q})$ 
on $P \oplus Q=P \cup Q$ such that $s \leq_{P \oplus Q} t$ 
if (a) $s,t \in P$ and $s \leq_{P} t$,
or (b) $s,t \in Q$ and $s \leq_{Q} t$,
or (c) $s \in P$ and $t \in Q$.
Then we have $\Ac(P \oplus Q)=\Ac(P) \cup \Ac(Q)$.
Let $P = \{p_1, \ldots, p_d\}$ and $Q = \{q_1, \ldots, q_d\}$.
Given a subset $W$ of $[d]$, we define the \textit{induced subposet} of $P$ on $W$
to be the partially ordered set $(P_W, \leq_{P_W})$ on $P_W=\{p_i \mid i \in W \}$
such that  $p_i \leq_{P_W} p_j$ if and only if $p_i \leq_{P} p_j$. 
For $W \subset [d]$,
we let $(\Delta_W(P,Q),\leq_{W})$ be the ordinal sum of $P_W$ and $Q_{\overline{W}}$, where $\overline{W}=[d] \setminus W$.
Note that
$|\Delta_{W}(P,Q)|=d$ 
and we have $\Ac(\Delta_W(P,Q))=\Ac(P_W) \cup \Ac(Q_{\overline{W}})$.
Let $W=\{i_1,\ldots,i_k\} \subset [d]$ and $\overline{W}=\{i_{k+1},\ldots,i_d\} \subset [d]$
with $W \cup \overline{W}=[d]$.
Then we have $\Delta_W(P,Q)=\{p_{i_1},\ldots,p_{i_k},q_{i_{k+1}},\ldots,q_{i_d}  \}$.
Also, we let $(R,\leq_R)$ be the partially ordered set on $R=\{r_1,\ldots,r_d\}$ such that $r_i \leq_R r_j$ 
if (a)  $i,j \in W$ and $p_i \leq_{W} p_j$,
or (b) $i,j \in \overline{W}$ and  $q_i \leq_{W} q_j$,
or (c) $i \in W,j \in \overline{W}$ and $p_i \leq_{W} q_j$.
We call a permutation $\sigma = i_1 i_2 \cdots i_d$ of $[d]$ a linear extension of $\Delta_W(P,Q)$,
if $\sigma$ is a linear extension of $R$,
and we write $e(\Delta_W(P,Q))$ for the number of linear extensions of $\Delta_W(P,Q)$,
i.e., $e(\Delta_W(P,Q))=e(R)$.
For $A \subset \Delta_W(P,Q)$,
we define the $(-1, 0, 1)$-vectors $\rho'(A) = \sum_{p_{i}\in A} \eb_{i}-\sum_{q_{j}\in A} \eb_{j}$
and we set $$\Cc'(\Delta_W(P,Q))=\text{conv}(\{ \rho'(A) \mid A \in \Ac(\Delta_W(P,Q)) \}).$$ 

Now, we recall properties of integral convex polytopes.
Let $\ZZ^{d \times d}$ denote the set of $d \times d$ integral matrices.
A matrix $A \in \ZZ^{d \times d}$ is {\em unimodular} if $\det (A) = \pm 1$.
Given integral convex polytopes $\Pc$ and $\Qc$ in $\RR^d$ of dimension $d$,
we say that $\Pc$ and $\Qc$ are {\em unimodularly equivalent}
if there exists a unimodular matrix $U \in \ZZ^{d \times d}$
and an integral vector ${\bf w}$, such that $\Qc=f_U(\Pc)+{\bf w}$,
where $f_U$ is the linear transformation in $\RR^d$ defined by $U$,
i.e., $f_U({\bf v}) = {\bf v} U$ for all ${\bf v} \in \RR^d$.
Clearly, if $\Pc$ and $\Qc$ are unimodularly equivalent, then
$\text{vol}(\Pc) = \text{vol}(\Qc)$, where $\text{vol}(\cdot)$ denotes the usual volume.

First, we will show the following lemma.
\begin{Lemma}
	\label{volumelemma}
	 Work with the same situation as above.
	 Then $\Cc'(\Delta_W(P,Q))$ and $\Cc(R)$ are unimodularly equivalent.
	  Moreover we have $$\textnormal{vol}(\Cc'(\Delta_W(P,Q)))=e(\Delta_W(P,Q))/d!.$$
\end{Lemma}
\begin{proof}
 Let $U=(u_{ij})_{1 \leq i,j \leq d} \in \ZZ^{d \times d}$ be a unimodular matrix such that 
	\begin{displaymath}
	u_{ij}=\left\{
	\begin{aligned}
	1, \ & \text{if } i=j \ \text{and} \  i \in W,\\
	-1, \ &\text{if } i=j \ \text{and} \  i \in \overline{W},\\
	0, \ & \text{if } i \neq j. 
	\end{aligned}
	\right.
	\end{displaymath}	
	Then $\Cc'(\Delta_W(P,Q)) = f_U(\Cc(R))$,
	where $(R,\leq_R)$ is the partially ordered set defined by the above.
	This says that $\Cc'(\Delta_W(P,Q))$ and $\Cc(R)$ are unimodularly equivalent.
	Hence since the volume of $\Cc(R)$ is equal to $e(R)/d!$,
	 We have 
$$
 \text{vol}(\Cc'(\Delta_W(P,Q)))=\text{vol}(\Cc(R))
	 =e(R)/d!
	 =e(\Delta_W(P,Q))/d!,
$$
	as desired.
\end{proof}

Let $\Pc \subset \RR^d$ be a integral convex polytope.
Then we write $V(\Pc)$ for the vertex set of $\Pc$,
and for $W \subset [d]$, we set
$$\Pc_W=\{(x_1,\ldots,x_d) \in \Pc \mid \text{if} \ i \in W, x_i \geq 0 \ \text{and if } j \in \overline{W}, x_j \leq 0 \},$$
$$V_W(\Pc)=\{(x_1,\ldots,x_d)\in V(\Pc) \mid \text{if} \ i \in W,\  x_i \geq 0 \ \text{and if} \ j \in \overline{W},\  x_j \leq 0 \}.$$

The following is the key proposition in this paper. 
\begin{Proposition}
	\label{prop}
		Let $(P,\leq_{P})$ and $(Q,\leq_{Q})$ be partially ordered sets on $P=\{p_1, \ldots, p_d\}$ and $Q = \{q_1, \ldots, q_d\}$.
		Then we have 
		$$\Gamma(\Cc(P),  \Cc(Q)) = \bigcup_{W \subset [d]}\Cc'(\Delta_W(P,Q)).$$
		In particular, for any subset $W \subset [d]$, we have 
		$$\Gamma(\Cc(P),  \Cc(Q))_W=\Cc'(\Delta_W(P,Q)).$$
\end{Proposition}

\begin{proof}
	For any $W \subset [d]$, we have 
	$$V_W(\Gamma(\Cc(P),  \Cc(Q)))=V(\Cc'(\Delta_W(P,Q)))\setminus\{(0,\ldots,0)\}$$
	since $\Ac(\Delta_W(P,Q))=\Ac(P_W)\cup \Ac(Q_{\overline{W}})$.
	Hence it follows that $$\Gamma(\Cc(P),  \Cc(Q))_W \supset \text{conv}(V_W(\Gamma(\Cc(P),  \Cc(Q))) \cup \{(0,\ldots,0)\})= \Cc'(\Delta_W(P,Q)).$$
	Moreover, we obtain $$\Gamma(\Cc(P),  \Cc(Q)) \supset \bigcup_{W \subset [d]}\Cc'(\Delta_W(P,Q)).$$
	
	We will show that for any $\xb,\yb \in V(\Gamma(\Cc(P),  \Cc(Q)))$ and $a,b \in \RR$ with $a+b=1,a \geq 0$ and $b \geq 0$, there exists $W \subset [d]$ such that $a\xb+b\yb \in \Cc'(\Delta_W(P,Q)) $.
	This shows that $\bigcup_{W \subset [d]}\Cc'(\Delta_W(P,Q))$ contains any edge of $\Gamma(\Cc(P),  \Cc(Q))$, hence, we have $$\Gamma(\Cc(P),  \Cc(Q)) \subset \bigcup_{W \subset [d]}\Cc'(\Delta_W(P,Q)).$$
	When $\xb,\yb \in \Cc(P)$ or $\xb,\yb \in (-\Cc(Q))$, it clearly follows.
	Let $$A_1=\{p_{i_1},\ldots,p_{i_{\ell}},p_{i_{\ell+1}},\ldots,p_{i_m} \}$$
	and
	$$A_2=\{q_{i_1},\ldots,q_{i_{\ell}},q_{i_{m+1}},\ldots,q_{i_n} \}$$
	be antichains of $\Ac(P)$ and $\Ac(Q)$, and we set $\xb=\rho(A_1)$ and $\yb=-\rho(A_2)$.
	We should show the case $a \geq b$.
	Let $W=\{i_{1},\ldots,i_m\} \subset [d]$ and $c=a-b$.
	Then $A'_1=\{p_{i_1},\ldots,p_{i_m} \}$, $A'_2=\{p_{i_{\ell+1}},\ldots,p_{i_m}\}$ and $A'_3=\{q_{i_{m+1}},\ldots, q_{i_n} \}$ are antichains of $\Delta_W(P,Q)$.
	We set $\xb'=\rho'(A'_1),\yb'=\rho'(A'_2)$ and $\zb'=\rho'(A'_3)$.
	Then we have $a\xb+b\yb=c\xb'+b\yb'+b\zb'$ and $c+2b=1$.
	Hence $a\xb+b\yb \in \Cc'(\Delta_W(P,Q)) $.
	
	Therefore, 
	we have
$$\Gamma(\Cc(P),  \Cc(Q)) = \bigcup_{W \subset [d]}\Cc'(\Delta_W(P,Q)),$$
	In particular,
		$$\Gamma(\Cc(P),  \Cc(Q))_W=\Cc'(\Delta_W(P,Q)).$$
		as desired.
\end{proof}

In this section, we give the formula of the volume of $\Gamma(\Cc(P), \Cc(Q))$ in terms of partially ordered sets $(P,\leq_P)$ and $(Q,\leq_Q)$.
In particular, the following theorem is immediate given Lemma \ref{volumelemma} and Proposition \ref{prop}.
\begin{Theorem}
	\label{volume}
	Let $(P,\leq_{P})$ and $(Q,\leq_{Q})$ be partially ordered sets on $P=\{p_1, \ldots, p_d\}$ and $Q = \{q_1, \ldots, q_d\}$.
Then we have
$$\textnormal{vol}(\Gamma(\Cc(P),  \Cc(Q)))=\sum\limits_{W \subset [d]}\cfrac{e(\Delta_W(P,Q))}{d!}.$$
\end{Theorem}

\begin{Remark}
In \cite{volOC}, it is shown that $\Gamma(\Cc(P),  \Cc(Q))$ has a regular unimodular triangulation.
Therefore, $d!\cdot\textnormal{vol}(\Gamma(\Cc(P),  \Cc(Q)))$ equals the number of maximal simplices in the triangulation.
We note that each linear extension of $e(\Delta_W(P,Q))$ corresponds to a maximal simplex in the triangulation.
In particular, we can prove Theorem \ref{volume} without Lemma \ref{volumelemma} and Proposition \ref{prop}.
\end{Remark}


We recall that
the volume of $\Gamma(\Cc(P),  \Cc(Q))$ equals that of $\Gamma(\Oc(P),  \Cc(Q))$, and
if $P$ and $Q$ have a common linear extension, 
then the volume of $\Gamma(\Cc(P),  \Cc(Q))$ equals that of $\Gamma(\Oc(P), \Oc(Q))$ (\cite{volOC}).
Hence by Theorem \ref{volume}, we obtain the following corollary.

\begin{Corollary}
	\label{oo}
	Let $(P,\leq_{P})$ and $(Q,\leq_{Q})$ be partially ordered sets on $P=\{p_1, \ldots, p_d\}$ and $Q = \{q_1, \ldots, q_d\}$.
	Then we have
	$$\textnormal{vol}(\Gamma(\Oc(P),  \Cc(Q)))=\sum\limits_{W \subset [d]}\cfrac{e(\Delta_W(P,Q))}{d!}.$$
	Moreover, if $P$ and $Q$ have a common linear extension, then we have
	$$\textnormal{vol}(\Gamma(\Oc(P),  \Oc(Q)))=\sum\limits_{W \subset [d]}\cfrac{e(\Delta_W(P,Q))}{d!}.$$
\end{Corollary}

\begin{Remark}
By the proof of Proposition \ref{prop}, for any $W \subset [d]$, $\Gamma(\Cc(P),  \Cc(Q))_W$ is an integral convex polytope.
However, $\Gamma(\Oc(P),  \Oc(Q))_W$ and $\Gamma(\Oc(P),  \Cc(Q))_W$ are not always integral.
In fact, let $P=\{p_1,p_2\}$ be a $2$-element chain with $p_1 \leq_{P} p_2$
and $Q=\{q_1,q_2\}$ a $2$-element chain with $q_1 \leq_{Q} q_2$.
Then for $W=\{1\}$, we know that $\Gamma(\Oc(P),  \Oc(Q))_W$ and $\Gamma(\Oc(P),  \Cc(Q))_W$ are not integral convex polytopes.
This means that we can not prove Corollary \ref{oo} by means of a proof similar to that of  Theorem \ref{volume}.
\end{Remark}

We end this section with a few examples.
\begin{Example}
	\label{ex2}
	{\rm
	Let $(P,\leq_P)$ and $(Q,\leq_Q)$ be partially ordered sets with the Haase diagrams shown in Figure 1.
	\newline
	\begin{picture}(400,100)(-80,5)
	\put(50,90){$P$:}
	\put(90,70){\circle*{5}}
	\put(110,30){\circle*{5}}
	\put(70,30){\circle*{5}}
	\put(74,68){$p_1$}
	\put(94,28){$p_3$}
	\put(54,28){$p_2$}
	\put(110,30){\line(-1,2){20}}
	\put(70,30){\line(1,2){20}}
	
	\put(150,90){$Q$:}
	\put(190,70){\circle*{5}}
	\put(210,30){\circle*{5}}
	\put(170,30){\circle*{5}}
	\put(174,68){$q_1$}
	\put(194,28){$q_3$}
	\put(154,28){$q_2$}
	\put(210,30){\line(-1,2){20}}
	\put(170,30){\line(1,2){20}}
	\put(125,5){Figure 1}
	\end{picture}
	\\ 
	Then $\Gamma(\Cc(P), \Cc(Q))$ is centrally symmetric, i.e.,
	for each facet $\Fc$ of $\Gamma(\Cc(P),  \Cc(Q))$, $-\Fc$ is a facet of $\Gamma(\Cc(P),  \Cc(Q))$.
	For each subset $W$ of $\{1,2,3\}$, the Haase diagram of $\Delta_W(P,Q)$ is presented in Figure 2.
	\newline
	\begin{picture}(400,100)(20,20)
	\put(50,90){$\Delta_{\{1,2,3\}}(P,Q)$:}
	\put(90,70){\circle*{5}}
	\put(110,30){\circle*{5}}
	\put(70,30){\circle*{5}}
	\put(74,68){$p_1$}
	\put(94,28){$p_3$}
	\put(54,28){$p_2$}
	\put(110,30){\line(-1,2){20}}
	\put(70,30){\line(1,2){20}}
	
	\put(150,90){$\Delta_{\{1,2\}}(P,Q)$:}
	\put(190,70){\circle*{5}}
	\put(190,50){\circle*{5}}
	\put(190,30){\circle*{5}}
	\put(174,68){$q_3$}
	\put(174,48){$p_1$}
	\put(174,28){$p_2$}
	\put(190,30){\line(0,1){20}}
	\put(190,50){\line(0,1){20}}
	
	\put(250,90){$\Delta_{\{1,3\}}(P,Q)$:}
	\put(290,70){\circle*{5}}
	\put(290,50){\circle*{5}}
	\put(290,30){\circle*{5}}
	\put(274,68){$q_2$}
	\put(274,48){$p_1$}
	\put(274,28){$p_3$}
	\put(290,30){\line(0,1){20}}
	\put(290,50){\line(0,1){20}}
	
	\put(350,90){$\Delta_{\{2,3\}}(P,Q)$:}
	\put(390,70){\circle*{5}}
	\put(410,30){\circle*{5}}
	\put(370,30){\circle*{5}}
	\put(374,68){$q_1$}
	\put(394,28){$p_3$}
	\put(354,28){$p_2$}
	\put(410,30){\line(-1,2){20}}
	\put(370,30){\line(1,2){20}}
	
	\end{picture}
	\\ 
	\newline
	\begin{picture}(400,100)(20,0)
	\put(50,90){$\Delta_{\{1\}}(P,Q)$:}
	\put(90,30){\circle*{5}}
	\put(110,70){\circle*{5}}
	\put(70,70){\circle*{5}}
	\put(74,28){$p_1$}
	\put(94,68){$q_3$}
	\put(54,68){$q_2$}
	\put(90,30){\line(-1,2){20}}
	\put(90,30){\line(1,2){20}}
	
	\put(150,90){$\Delta_{\{2\}}(P,Q)$:}
	\put(190,70){\circle*{5}}
	\put(190,50){\circle*{5}}
	\put(190,30){\circle*{5}}
	\put(174,68){$q_1$}
	\put(174,48){$q_3$}
	\put(174,28){$p_2$}
	\put(190,30){\line(0,1){20}}
	\put(190,50){\line(0,1){20}}
	
	\put(250,90){$\Delta_{\{3\}}(P,Q)$:}
	\put(290,70){\circle*{5}}
	\put(290,50){\circle*{5}}
	\put(290,30){\circle*{5}}
	\put(274,68){$q_1$}
	\put(274,48){$q_2$}
	\put(274,28){$p_3$}
	\put(290,30){\line(0,1){20}}
	\put(290,50){\line(0,1){20}}
	
	\put(350,90){$\Delta_{\emptyset}(P,Q)$:}
	\put(390,70){\circle*{5}}
	\put(410,30){\circle*{5}}
	\put(370,30){\circle*{5}}
	\put(374,68){$q_1$}
	\put(394,28){$q_3$}
	\put(354,28){$q	_2$}
	\put(410,30){\line(-1,2){20}}
	\put(370,30){\line(1,2){20}}
	\put(220,5){Figure 2}
	\end{picture}
	\\
	Hence we have
	$$\textnormal{vol}(\Gamma(\Cc(P),  \Cc(Q)))=4\times\cfrac{1}{6}+4\times \cfrac{2}{6}=2.$$ 
}
\end{Example}

\begin{Example}
	\label{ex}
	{\rm
		Let $P=\{p_1,\ldots,p_d \}$ be a $d$-element antichain
		and $Q=\{q_1,\ldots,q_d \}$ a $d$-element chain with $q_1 \leq_Q \cdots \leq_Q q_d$.	
		For $W \subset [d]$, we will compute the volume of $\Cc'(\Delta_W(P,Q))$.
		We set $W=\{1,\ldots,k\}$.
		Then $P_W$ is a $k$-element antichain and $Q_{\overline{W}}$ is a $(d-k)$-element chain.
		Hence we have $$\Cc'(\Delta_W(P,Q))=\textnormal{conv}(\{[0,1]^k\times\{0\}^{d-k},-\eb_{k+1},\ldots,-\eb_d\})$$
		 and $\textnormal{vol}(\Cc'(\Delta_W(P,Q)))=k!/d!$.
		Therefore, we obtain
		\begin{displaymath}
		\begin{aligned}
		\textnormal{vol}(\Gamma(\Cc(P), \Cc(Q)))&=\sum\limits_{k=0}^{d}\binom{d}{k}\cfrac{k!}{d!}\\
		&=\sum\limits_{k=0}^{d}\cfrac{1}{k!}.
		\end{aligned}
		\end{displaymath}
		Moreover, by Corollary \ref{oo}, we have 
			$$	\textnormal{vol}(\Gamma(\Oc(P),  \Oc(Q)))=\textnormal{vol}(\Gamma(\Oc(P),  \Cc(Q)))=\sum\limits_{k=0}^{d}\cfrac{1}{k!}.$$
		
	For a positive integer $d$, we write $a(d)$ for the total number of arrangements of a $d$-element set.
	Then  we have	$$\textnormal{vol}(\Gamma(\Cc(P),  \Cc(Q)))=\cfrac{a(d)}{d!}.$$
}
\end{Example}

\section{facets and dual polytopes of twinned chain polytopes}
In this section, we will compute the facet-supporting hyperplanes and dual polytopes of twinned chain polytopes.
We begin by recalling these features for the chain polytopes which were originally studied in \cite{Stanley}.
Let $(P,\leq_P)$ be a partially ordered set on $P=\{p_1,\ldots,p_d\}$.
Then there are two types of the facet-supporting hyperplanes for the chain polytope $\Cc(P)$:
\begin{itemize}
	\item for each element $p_i$ of  $P$, $x_i = 0$,
	\item for each maximal chain $C$ of  $P$, $\sum_{p_i \in C}x_i = 1$.
\end{itemize}
We write $\Mc(P)$ for the set of maximal chains of $P$.
Then the number of facets of $\Cc(P)$ equals $|\Mc(P)|+d$.

The next lemma follows immediately from Lemma \ref{volumelemma}.
\begin{Lemma}
	\label{facet}
	Let $(P,\leq_{P})$ and $(Q,\leq_{Q})$ be partially ordered sets on $P=\{p_1, \ldots, p_d\}$ and $Q = \{q_1, \ldots, q_d\}$,
	 and let $W \subset [d]$.
	Then there are three types of the facet-supporting hyperplanes for $\Cc'(\Delta_W(P,Q)))$:
	\begin{itemize}
		\item for each element $p_i$ of $\Delta_W(P,Q)$, $x_i = 0$,
		\item for each element $q_j$ of  $\Delta_W(P,Q)$, $-x_j = 0$, 
		\item for each maximal chain $C$ of $\Delta_W(P,Q)$, $\sum_{p_i \in C}x_i-\sum_{q_j \in C}x_j = 1$.
	\end{itemize}
\end{Lemma}

In this section, we characterize  the facet-supporting hyperplanes of $\Gamma(\Cc(P), \Cc(Q))$ in terms of partially ordered sets $(P,\leq_{P})$ and $(Q,\leq_{Q})$.
Namely, we prove the following theorem.
\begin{Theorem}
	\label{facetth}
			Let $(P,\leq_{P})$ and $(Q,\leq_{Q})$ be partially ordered sets on $P=\{p_1, \ldots, p_d\}$ and $Q = \{q_1, \ldots, q_d\}$.
		The facet-supporting hyperplanes for $\Gamma(\Cc(P),  \Cc(Q))$ are given as
			$$\sum_{p_i \in C}x_i-\sum_{q_j \in C}x_j = 1$$
		for each  $W \subset [d]$ and for each maximal chain $C$ of  $\Delta_W(P,Q)$.		
		Moreover, the number of facets of $\Gamma(\Cc(P),  \Cc(Q))$ equals $|\bigcup_{W \subset [d]}\Mc(\Delta_W(P,Q))|$.
\end{Theorem}
\begin{proof}
	We let $W$ be a subset of $[d]$ and let $C$ be a maximal chain of $\Delta_W(P,Q)$.
	Then by Lemma \ref{facet}, $\Fc_C=\Hc_C \cap \Cc'(\Delta_W(P,Q))$ is a facet of $\Cc'(\Delta_W(P,Q))$ 
	, where $\Hc_C$ is the hyperplane 
	$$\{(x_1,\ldots,x_d) \in \RR^d \mid \sum_{p_i \in C}x_i-\sum_{q_j \in C}x_j = 1\}$$
	in $\RR^d$.
	We let $\yb=(y_1,\ldots,y_d)$ be an interior point of $\Fc_C$.
	Then by Lemma \ref{facet}, we know $y_i > 0$ if $i \in W$ and $y_j <0$ if $j \in \overline{W}$.
	Hence for any $W' \subset [d]$ with $W \neq W'$, we have $\yb \notin \Cc'(\Delta_{W'}(P,Q))$.
	Therefore, it follows that $\yb$ does not belong to the interior of $\Gamma(\Cc(P), \Cc(Q))$.
	By Proposition \ref{prop}, $\Hc_C \cap \Gamma(\Cc(P), \Cc(Q))$ is a facet of $\Gamma(\Cc(P),  \Cc(Q))$.
	
	Since $\Gamma(\Cc(P),  \Cc(Q))$ is Gorenstein Fano, 
	the supporting hyperplane of each facet of $\Gamma(\Cc(P), \Cc(Q))$ is of the form
	$a_1 x_1+\cdots+a_d x_d=1$
	with each $a_i \in \ZZ$.
	Hence the facet-supporting hyperplanes for $\Gamma(\Cc(P),  \Cc(Q))$ are given as
	$$\sum_{p_i \in C}x_i-\sum_{q_j \in C}x_j = 1$$
	for each  $W \subset [d]$ and for each maximal chain $C$ of  $\Delta_W(P,Q)$,
		as desired.
\end{proof}

\begin{Remark}
	For some partially ordered sets $(P,\leq_P)$ and $(Q,\leq_Q)$ with $|P|=|Q|=d$, we have $$\sum_{W \subset [d]}|\Mc(\Delta_W(P,Q))| \neq |\bigcup_{W \subset [d]}\Mc(\Delta_W(P,Q))|.$$
	For instance, let $P=\{p_1,p_2,p_3\}$ and $Q=\{q_1,q_2,q_3\}$ be $3$-element antichains.
	For $W_1=\{1\}$, $C_1=\{p_1,q_3\}$ is a maximal chain of $\Delta_{W_1}(P,Q)$.
	Then for $W_2=\{1,2\}$, $C_1$ is also a maximal chain of $\Delta_{W_2}(P,Q)$.
	Hence we have $$\sum_{W \subset [d]}|\Mc(\Delta_W(P,Q))| > |\bigcup_{W \subset [d]}\Mc(\Delta_W(P,Q))|.$$	
\end{Remark}

Let $\Pc \subset \RR^d$ be a Gorenstein Fano polytope of dimension $d$. 
	Then a point $\ab \in \RR^d$ is a vertex of $\Pc^\vee$ if 
	and only if $\Hc \cap \Pc$ is a facet of $\Pc$,
	where $\Hc$ is the hyperplane
	$$\left\{ \xb \in \RR^d \ | \ \langle \ab, \xb \rangle =1 \right\}$$
	in $\RR^d$ (\cite[Corollary 35.6]{HibiRedBook}).
Hence by Theorem \ref{facet}, we obtain the following Corollary.

\begin{Corollary}
	\label{cor}
			Let $(P,\leq_{P})$ and $(Q,\leq_{Q})$ be partially ordered sets on $P=\{p_1, \ldots, p_d\}$ and $Q = \{q_1, \ldots, q_d\}$.
		Then we have $$V(\Gamma(\Cc(P), \Cc(Q))^\vee)=\bigcup_{W \subset [d]}\{\rho'(C) \in \RR^d \mid C \in \Mc(\Delta_W(P,Q)) \}.$$
	Namely, 
	$$\Gamma(\Cc(P), \Cc(Q))^\vee=\textnormal{conv}\left(\bigcup_{W \subset [d]}\{\rho'(C) \in \RR^d \mid C \in \Mc(\Delta_W(P,Q))\}\right).$$
\end{Corollary}

We end this section with a pair of examples demonstrating Theorem \ref{facetth} and Corollary \ref{cor}.
\begin{Example}
	{\rm
		Let $(P,\leq_P)$ and $(Q,\leq_Q)$ be partially ordered sets in Example \ref{ex}.
		We fix $W=\{i_1,\ldots,i_k\} \subset [d]$.
		Then we have
		$$\Mc(\Delta_W(P,Q))=\{ \{p_{i_s},q_{i_{k+1}},\ldots,q_{i_d}\} \mid 1 \leq s \leq k \}$$
		and $|\Mc(\Delta_W(P,Q))|=k$.
		Hence 
		\begin{displaymath}
		\begin{aligned}
		|\bigcup_{W \subset [d]}\Mc(\Delta_W(P,Q))|&=\sum\limits_{k=1}^{d}\binom{d}{k}k+1\\
		&=d\cdot 2^{d-1}+1.
		\end{aligned}
		\end{displaymath}
	}
\end{Example}

\begin{Example}
	{\rm
		Let $(P,\leq_P)$ and $(Q,\leq_Q)$ be partially ordered sets in Example \ref{ex2}.
	Then by Corollary \ref{cor}, the vertices of $\Gamma(\Cc(P),  \Cc(Q))^\vee$ are the following:
	$$\pm(1,1,0),\pm(1,0,1),\pm(1,-1,0),\pm(1,1,-1),\pm(1,-1,1),\pm(1,0,-1).$$
    Moreover, there do not exist partially ordered sets $(P',\leq_{P'})$ and $(Q',\leq_{Q'})$ with $|P'|=|Q'|=3$ such that $\Gamma(\Cc(P),  \Cc(Q))^\vee$ and $\Gamma(\Cc(P'), \Cc(Q'))$ are unimodularly equivalent.
    Indeed, since $\Gamma(\Cc(P),  \Cc(Q))^\vee$ is centrally symmetric and the number of its vertices equals $12$,
    each of $P'$ and $Q'$ needs to have just $7$ antichains.
    However, there exists no $3$-element partially ordered set which has just $7$ antichains. 
}
\end{Example}

\begin{acknowledgement}
	The author would like to thank anonymous referees for reading the manuscript carefully.
\end{acknowledgement}

\end{document}